\documentclass[a4paper,12pt, leqno]{amsart}
\usepackage{amsmath,amssymb,amsthm,amsfonts,amstext,amscd}
\usepackage[margin=2.5cm]{geometry}
\usepackage[english]{babel}
\usepackage[ansinew]{inputenc}
\usepackage[T1]{fontenc}
\usepackage{lmodern}
\usepackage{newlfont}
\usepackage{longtable}
\usepackage[all]{xy}
\usepackage{color}

\newcommand{\nc}{\newcommand}

\theoremstyle{plain}
\newtheorem{theorem}{\sc Theorem}[section]
\newtheorem{thm}[theorem]{\sc Theorem}
\newtheorem{lem}[theorem]{\sc Lemma}

\newtheorem{prop}[theorem]{\sc Proposition}
\newtheorem{cor}[theorem]{\sc Corollary}

\newtheorem{rem}[theorem]{\sc Remark}
\newtheorem{ex}[theorem]{\sc Example}

\theoremstyle{definition}

\renewcommand{\refname}{\center{REFERENCES}}

\newcommand{\vfi}{\varphi}

\newcommand{\N}{\mathbb{N}}
\newcommand{\Z}{\mathbb{Z}}

\newcommand{\Nabla}{\bigtriangledown}

\nc{\GG}{\mathfrak{G}}

\nc{\LL}{\mathcal{L}}
\nc{\JJ}{\mathcal{J}}

\nc{\fN}{\mathfrak{N}}
\nc{\fH}{\mathfrak{H}}
\nc{\fF}{\mathfrak{F}}
\nc{\cH}{\mathcal{H}}
\nc{\cG}{\mathcal{G}}
\nc{\cN}{\mathcal{N}}
\nc{\cK}{\mathcal{K}}
\nc{\cT}{\mathcal{T}}

\nc{\pgs}{\mathfrak{p}\mathfrak{g}\mathfrak{s}}
\nc{\pcr}{\mathfrak{p}\mathfrak{c}\mathfrak{r}}
\nc{\wg}{\widehat{g}}
\nc{\wh}{\widehat{h}}
\nc{\wx}{\widehat{x}}
\nc{\wy}{\widehat{y}}
\nc{\wk}{\widehat{k}}

\nc{\Char}{\operatorname{char}}
\nc{\Ker}{\operatorname{Ker}}
\nc{\Imm}{\operatorname{Im}}
\nc{\Aut}{\operatorname{Aut}}
\nc{\Cl}{\operatorname{Cl}}
\nc{\Orb}{\operatorname{Orb}}
\nc{\noreq}{\trianglelefteq}
\nc{\lcm}{\operatorname{lcm}}

\advance\textheight by \topskip
\title[Structural and Closure Properties for the $q$-Tensor Product]
{Some Structural and Closure Properties of an Extension of the $q$-Tensor Product of Groups, $q \geq 0$}
\author[Dias]{Ivonildes Ribeiro Martins Dias}
\address{Institute of Mathematics and Statistics, 
Universidade Federal de Goi\'as, Goi\^ania-GO, 74001-970 Brazil}
\email{ivonildes@ufg.br}
\author[Rocco]{Nora\'i Romeu Rocco*}
\address{Departamento de Matem\'atica-IE, Universidade de Bras\'ilia,
Bras\'ilia-DF, 70910-900 Brazil}
\email{norai@unb.br}
\thanks{(*) The author acknowledges partial financial support from FAPDF, Brazil, during the preparation of this
work.}
\author[Rodrigues]{Eunice C\^andida Pereira Rodrigues}
\address{Institute of Exact and Natural Sciences, Universidade Federal de Rondon\'opolis, 78730-614 Rondon\'opolis-MT, Brazil}
\email{eunicecpr@hotmail.com}
\date{\today}

\begin{document}

\begin{abstract}
In this work we study some structural properties of the group $\eta^q(G, H)$, $q$ a non-negative integer, which is an extension of the $q$-tensor product $G \otimes^q H$, where $G$ and $H$ are normal subgroups of some group $L$. We establish by simple arguments some closure properties of $\eta^q(G, H)$ when $G$ and $H$ belong to certain Schur classes. This extends similar results concerning the case $q = 0$ found in the literature. Restricting our considerations to the case $G = H$, we compute the $q$-tensor square $D_n \otimes^q D_n$ for $q$ odd, where $D_n$ denotes the dihedral group of order $2n$. Upper bounds to the exponent of $\eta^q(G, G)$ are also established for nilpotent groups $G$ of class $\leq 3$, which extend to all $q \geq 0$ similar bounds found by Moravec in \cite{Moravec}.    
\end{abstract}

\maketitle

\section{Introduction} 

Let $G$ and $H$ be groups each of which acts upon the other (on the right)  
and upon themselves by conjugation, in a compatible way, that is, 
\begin{equation}   \label{eq:0}
g^{\left( h^{g_1} \right) } = \left( \left( g^{g^{-1}_1}  \right) ^h \right) ^{g_1} \; \; \mbox{and} \; \; h^{\left( g^{h_1}\right) } =
\left( \left( h^{h_1^{-1}} \right) ^g \right) ^{h_1},
\end{equation}
for all $g, g_1 \in G$ and $h, h_1 \in H$. In this situation, the \textit{non-abelian tensor product} $G \otimes H$ of $G$ and $H$, as defined by Brown and Loday in \cite{BL}, is the group generated by the symbols $g \otimes h$, where $g \in G$ and $h \in H$, subject to the defining relations 
\begin{equation} 
gg_{1}\otimes h = (g^{g_{1}}\otimes h^{g_{1}})(g_{1}\otimes h) \quad \text{and} \quad g\otimes hh_{1}= (g\otimes h_{1})(g^{h_{1}}\otimes h^{h_{1}}),
\end{equation}
where $g,g_1 \in G$ and $h,h_1 \in H$.

Brown and Loday \cite{BL} gave a topological significance for the non-abelian tensor product of groups. They used it to describe the third relative homotopy group of a triad as a non-abelian tensor product of the second relative homotopy groups of appropriate subspaces (see also \cite{BRV}). 
When $G = H$ and all actions are by conjugation in $G$, then the
group $G \otimes G$ is called the \textit{non-abelian tensor square} of $G$. The commutator map induces a homomorphism 
$\kappa: G \otimes G \to G$, such that $g \otimes h \mapsto [g, h],$ for all $g, h \in G,$ whose kernel is usually denoted by $J_2(G).$ Its topological interest is given by the isomorphism (Cf. \cite{BL1,BL}):
\[\pi_3SK(G,1) \cong J_2(G),\]
where $\pi_3 SK(G,1)$ is the third homotopy group of the suspension of the Eilenberg-MacLane space $K(G,1)$. 

Non-abelian tensor products of groups have been studied by a number of authors. In \cite{Rocco1} the second author derived some  properties of the non-abelian tensor square of a group $G$ via its embedding in a larger group, $\nu(G)$, defined as follows. Let 
$G^\varphi$ be an isomorphic copy of $G$ via an isomorphism 
$\varphi: G \to G^\varphi$, such that $g \mapsto g^{\varphi}$, for all $g \in G$. Then, $\nu(G)$ is defined to be the group 
\[ \nu(G):= \langle G \cup G^\varphi \; \vert \; [g,h^\varphi]^k =[g^k,(h^k)^\varphi]=[g,h^\varphi]^{k^\varphi}\rangle, \; \text{for all} \; g, h, k \in G.\]

Besides its intrinsic group theoretical interest, the motivation in introducing $\nu(G)$ is that its subgroup $[G, G^\varphi]$ is naturally isomorphic with the non-abelian tensor square, $G \otimes G$. Independently, Ellis e Leonard \cite{EL} introduced a similar construction. 

Following \cite{Rocco1} and \cite{EL}, Nakaoka \cite{Nak} extended the operator $\nu$ to an operator $\eta$, for the case of two groups $G$ and $H$ acting compatibly on one another. To this end it is considered an isomorphic copy $H^\varphi$ of $H$, where $\varphi : h \mapsto h^\varphi$, for all $h \in H$. For any $x, x_1 \in G$ and $y, y_1 \in H,$ set
\begin{align} 
s_1(x, y, x_1) &= [x, {y}^{\varphi}]^{x_1} \cdot [{x}^{x_1},({y}^{x_1})^{\varphi}]^{-1}, \label{rel:s1} \\
s_2(x, y, y_1) &= [x, {y}^{\varphi}]^{y^{\varphi}_1} \cdot [{x}^{y_1},({y}^{y_1})^{\varphi}]^{-1}. \label{rel:s2} 
\end{align}
Let $S_1 = \{s_1(x, y, x_1) \, \vert \, x, x_1 \in G, \; y \in H \}$, $S_2 = \{s_2(x, y, y_1) \, \vert \, x \in G, \; y, y_1 \in H \}$ and $S = S_1 \cup S_2$. Then the group $\eta(G, H)$ is defined by (Cf. \cite{Nak}):
\begin{equation} \label{def:eta}
      \eta(G, H) = \langle G, H^{\varphi} \, | \, S  \rangle,
     \end{equation}
the factor group of the free product $G \ast H^{\varphi}$ by its normal subgroup generated by $S$. It follows from \cite[Proposition 1.4]{GH} that the map $g \otimes h \mapsto [g, h^{\vfi}]$ gives rise to an isomorphism from $G \otimes H$ onto the subgroup $[G, H^{\vfi}] \triangleleft \eta(G,H)$. When $G = H$ and all actions are by conjugation then $\eta(G, H)$ becomes the group $\nu(G).$ 

Ellis and Rodr\'iguez-Fern\'andez \cite{ER89}, Brown \cite{Brown}, and Conduch\'e and Rodr\'iguez-Fern\'andez \cite{Cond} started the investigation of a modular version of the non-abelian tensor product. In \cite{Ellis} Ellis considered the so called $q$-tensor product $G \otimes^q H$, of $G$ and $H$, where $q$ is a non-negative integer, in the case where $G$ and $H$ are normal subgroups of a larger group $L$ and all actions are by conjugation in $L$. In this situation the q-tensor product, $G {\otimes}^q H$, is defined to be the group generated by the symbols $g \otimes h$ and $\widehat{k}$, for $g \in G$, $h \in H$ and $k \in G \cap H$, subject to the following relations (for all $g, g_1 \in G$, $h, h_1 \in H$ and $k, k_1 \in G \cap H$), where for elements
$x, \, y \in L$ the conjugate of $x$ by $y$ is written $x^y =
y^{-1}xy$ and the commutator of $x$ and $y$ is $[x, \, y] =
x^{-1}x^y$:
  \begin{align} 
  g \otimes hh_1 &= (g \otimes h_1)(g^{h_1} \otimes h^{h_1}); \label{rel1} \\
  gg_1 \otimes h &= (g^{g_1} \otimes h^{g_1})(g_1 \otimes h); \label{rel2} \\
  (g \otimes h)^{\widehat{k}} &= g^{(k^q)} \otimes h^{(k^q)}; \label{rel3} \\
 \widehat{k k_1} &= \widehat{k} \displaystyle{\prod_{i=1}^{q-1}}(k \otimes
  (k^{-i}_1)^{k^{q-1-i}})\widehat{k_1}; \label{rel4} \\
  [\widehat{k},\widehat{k_1}] &= k^q \otimes k^q_1; \label{rel5} \\
 \widehat{[g,h]} &= (g \otimes h)^q. \label{rel6} 
  \end{align}
For $q=0$ the $0$-tensor product $G {\otimes}^0 H$ is
just the non-abelian tensor product $G \otimes H$ (cf.
\cite{BL}), that is, the group generated by the symbols $g
\otimes h$, for $g \in G$, $h \in H$, subject to the relations
$(\ref{rel1})$ and $(\ref{rel2})$ only. When $G = H$ then we get the $q$-tensor square, $G \otimes^q G$. 

In \cite{Brown} Brown showed that if $G$ 
is a $q$-perfect group, that is, $G$ is generated by commutators and  $q$-th powers, then the (unique) universal $q$-central extension of $G$ is isomorphic with 
\[1\to H_2(G,\mathbb Z_q) \to G \otimes^q G \to G \to 1.\] 

In this paper we address some structural properties of the group $\eta^q(G, H),$ an extension of $G \otimes^q H$ by $G \times H$ which also generalizes in a certain sense the group $\eta(G, H)$ in the particular situation where $G$ and $H$ are normal subgroups of a larger group $L$. The group $\eta^{q}(G, H)$ first appeared in \cite{Ellis1} using a slightly different approach.  

Notation is fairly standard (see for instance \cite{Rob}); as usual we write $x^y$ to mean the conjugate $y^{-1} x y$ of $x$ by $y$; the commutator of $x$ and $y$ is then $[x, y] = x^{-1} y^{-1} x y$. Our commutators are \textit{left normed}, that is,  $[x, y , z] = [[x, y], z]$. 

The paper is organized as follows. In Section 2 we briefly describe the group 
$\eta^q(G, H)$ and treat of some basic structural results. In Section 3 we prove Theorem~\ref{thm:finiteness}, where we address some closure properties of $\eta^q(G,H)$, extending similar results concerning $\nu^q(G)$ and the $q$-tensor square found in \cite{BR} and elsewhere. In Section 4 we concentrate on polycyclic groups and present some computations. Finally, in Section 5 we prove Theorem~\ref{thm:expclass3}, were we establish upper bounds to the exponent of $G \otimes^q G$ for nilpotent groups of at most class 3, extending to all $q \geq 0$ similar bounds found by Moravec~\cite{Moravec} in the case $q = 0$. 

{\bf Acknowledgements:} 
The authors are very grateful to Raimundo Bastos for the interesting discussions and suggestions on the best approach to these results.

\section{Basic Structural Results} \label{sec:2}

We begin this section by giving a brief description of the group 
$\eta^q(G, H).$ To this end we assume that $G$ and $H$ are normally embedded in a larger group $L$ and that all actions are by conjugation in $L$. For $q \geq 1$ let ${\cal K} = G\cap H$ and let $\widehat{\cal K}=\left\{\widehat k \vert k\in {\cal K}\right\}$ be a set of symbols, one for each element of ${\cal K}$ (for $q=0$ we set  $\widehat{\cal K}=\emptyset$, the empty set). Let $F(\widehat{\cal K})$ be the free group on  
$\widehat{\cal K}$ and $\eta(G,H)*F(\widehat{\cal K})$ be the free product of $\eta(G,H)$ and $F(\widehat{\cal K})$. Since 
$G$ and $H^\varphi$ are embedded into $\eta(G,H)$, we shall identify the elements of $G$ (respectively of $H^\varphi$) with their respective images in $\eta(G,H)*F(\widehat{\cal K})$. Let $J$ be the normal closure in  
$\eta(G,H)*F(\widehat{\cal K})$ of the following elements, for all $k, k_1 \in {\cal K}, g \in G$ and $h\in H$:
\begin{eqnarray} \label{reta1}
g^{-1}\widehat{k}g(\widehat{k^g})^{-1};
\end{eqnarray}
\begin{eqnarray} \label{reta2}
(h^\varphi)^{-1}\widehat{k}{h^\varphi}(\widehat{k^h})^{-1};
\end{eqnarray}
\begin{eqnarray} \label{reta3}
(\widehat{k})^{-1}\left[ g,h^\varphi \right]\widehat{k}\left[ g^{k^q},(h^{k^q})^\varphi \right]^{-1};
\end{eqnarray}
\begin{eqnarray}\label{reta4}
(\widehat {k})^{-1}\widehat{kk_{1}} (\widehat{k_1})^{-1}\left(\displaystyle \prod _{i=1}^{q-1} \left[k,(k_{1}^{-i})^\varphi \right]^{k^{q-1-i}}\right)^{-1}; 
\end{eqnarray}
\begin{eqnarray} \label{reta5}
\left[ \widehat{k}, \widehat{k_{1}}\right]\left[k^{q},({k_{1}}^{q})^\varphi \right]^{-1};
\end{eqnarray}
\begin{eqnarray}\label{reta6}
\widehat{\left[ g,h \right]} \left[ g, h^\varphi \right] ^{-q}.
\end{eqnarray}

According to \cite{BR} (see also \cite{Ellis}), the group $\eta^q(G,H)$ is then defined to be the factor group
\begin{equation} \label{def:etaq} 
\eta^q(G,H):=\left(\eta(G,H)*F({\cal K})\right)/J.
\end{equation}

For $q=0$ the set of relations from (\ref{reta1}) to (\ref{reta6}) is empty; in this case we have 
$\left(\eta(G,H)*F(\widehat{\cal K})\right)/J\cong \eta(G,H)$. Also, for $G=H=L$ we get that $\eta^q(G,G) \cong \nu^q(G)$, which becomes the group $\nu(G)$ if $q=0$. 

There is an epimorphism $\rho: \eta^q(G,H) \to GH$, $g \mapsto g$, $h^\varphi \mapsto h$, $\widehat k \mapsto k^q$. On the other hand, the immersion of $G$ into  $\eta(G,H)$ induces a homomorphism $i: G \to \eta^q(G,H)$. We have that 
$g^{i\rho}=g$ and thus $i$ is injective. Similarly, the immersion of $H^\varphi$ into $\eta(G,H)$ induces a homomorphism 
$j: H^\varphi \to \eta^q(G,H)$. Hence, the elements $g \in G$ and $h^\varphi \in H^\varphi$ are identified with their respective images $g^i$ and $(h^\varphi)^j$ in $\eta^q(G,H)$. 

We write $K$ to denote the subgroup of $\eta^q(G,H)$ generated by the images of 
$\widehat{{\cal K}}$. By relations 
\eqref{reta1} and \eqref{reta2}, $K$ is normal in $\eta^q(G,H)$ and, by relations \eqref{rel:s1}, \eqref{rel:s2} and \eqref{reta3}, the subgroup $T:=[G, H^\varphi]$ is normal in $\eta^q(G,H).$ Consequently, $\Upsilon^q(G, H):=[G, H^\varphi] K$ is a normal subgroup of $\eta^q(G,H)$. 

By the above considerations we obtain 
\begin{eqnarray} \label{eq:structure}
\eta^q(G, H) = (\Upsilon^q(G, H) \cdot G) \cdot H^\varphi, 
\end{eqnarray}
where the dots indicate (internal) semi-direct products. 

Besides its intrinsic interest as a group theoretical construction, one of the main motivations to introduce and study the group $\eta^q(G, H)$ is the canonical ``hat'' (power) and commutator approach to the $q$-tensor product 
$G \otimes^q H$.   

In effect, an adaptation of the proof of \cite[Proposition 2.9]{BR} can be easily carried out (see also \cite[Theorem 8]{Ellis}) to give us the following: 
\begin{prop} \label{prop:iso1}
There is an isomorphism $\Upsilon^q(G, H)  \cong G \otimes^q H$ such that 
$[g, h^\varphi] \mapsto g \otimes h$ and $\hat{k} \mapsto \hat{k}$, for all $g \in G, h \in H$ and $k \in \mathcal{K}$. 
\end{prop}
This approach provides us not only with more psychological comfort but it also brings computational advantages by treating tensors as commutators in a larger group (see for instance, \cite{EL}, \cite{Rocco}, \cite{BR}, \cite{EN}, \cite{RR}, \cite{DR}). 

The $q$-exterior product $G \wedge^q H$ is defined to be the quotient of $G \otimes H$ by its (central) subgroup 
$\Nabla^q(G, H) :=  \langle k \otimes k \; \vert \; k \in \mathcal{K} \rangle.$ According to our approach we write  
\begin{eqnarray} \label{eq:ext}
\Delta^q(G, H) := \langle [k, k^\varphi] \, \vert \, k \in \mathcal{K} \rangle \qquad \text{and} \qquad \tau^q(G, H):= \frac{\eta^q(G, H)}{\Delta^q(G, H)},
\end{eqnarray}
so that 
\begin{eqnarray} \label{eq:ext1}
 G \wedge^q H \cong \frac{\Upsilon^q(G, H)}{\Delta^q(G, H)}.
\end{eqnarray}

In the following Lemma we collect some basic consequences of the defining relations of 
$\eta^q(G,H)$; their proofs can be easily adapted from \cite[Lemma 2.4]{BR} and are omitted. 

\begin{lem} \cite[Lemma 2.4]{BR} \label{lem:basic} 
Let $G$ and $H$ be normal subgroups of a group $L$ and $q \geq 0$. Then the following relations hold in $\eta^q(G,H)$, for all $g, g_1 \in G, \; h, h_1 \in H \; \text{and} \; k, k_1 \in {\cal K}$.
	\begin{itemize}
		\item[(i)] $[g,h^\varphi]^{[g_1, h_1^\varphi]}=[g, h^\varphi]^{[g_1, h_1]}$; 
		\item[(ii)]$[g, h^\varphi, h^\varphi_1]=[g,h, h^\varphi_1]$; \ $[g_1, [g, h^\varphi]] = [g_1, [g,h]^\varphi].$ In particular, $[g, h^\varphi, k^\varphi]=[g,h, k^\varphi]=[g, h^\varphi, k] = [[g, h]^\varphi, k]$;
		\item[(iii)] If $k \in {\cal K}'$ (or if $k_1 \in {\cal K}'$) then $[k,k_1^\varphi][k_1,k^\varphi]=1$;
		\item[(iv)] $[\widehat{k},[g,h]]=[\widehat{k},[g,h^\varphi]] = [\widehat{k},[g,h]^\varphi] = 
		[k^q, [g, h^\varphi]] = [(k^q)^\varphi, [g, h^\varphi]]$;
		\item[(v)] $[\widehat{k}, h^\varphi] = [k^q, h^\varphi], \; [g, \widehat{k}] = [g, (k^q)^\varphi]$; 
		\item[(vi)] If $[k, k_1]=1$ then $[k,k_1^\varphi]$ and 
		$[k_1, k^\varphi]$ are central elements in $\eta^q(G,H)$ and they have the same finite order dividing $q$. If in addition $k$ and $k_1$ are torsion elements of orders $o(k)$ and $o(k_1)$, respectively, then the order of $[k, k_1^\varphi]$ divides $\gcd(q, o(k),o(k_1))$;
		\item[(vii)] $[k,k^\varphi]$ is central in  $\eta^q(G,H)$, for all $k\in {\cal K}$;
		\item[(viii)]$[k,k_1^\varphi][k_1, k^\varphi]$ is central in  
		$\eta^q(G,H)$;
		\item[(ix)] $[k, k^\varphi] = 1$, for all $k\in {\cal K}'$;
		\item[(x)] If $[k,g] = 1 = [k, h]$ then 
		$[g,h,k^\varphi]= 1 = [[g,h]^\varphi, k]$.
	\end{itemize}
 \end{lem}
 
The next corollary extends \cite[Corollary 2.5]{BR} to $\eta^q(G, H)$.
\begin{cor} \cite[Corollary 2.5]{BR} \label{cor:central}
If $[G, H] = 1$ then $\Upsilon^q(G,H)$ is a central subgroup of $\eta^q(G,H)$. Furthermore, in this case we have
\begin{eqnarray} \label{eq:abelian}
\Upsilon^q(G, H) \cong \frac{G}{G'G^q} \otimes_{\Z} \frac{H}{H'H^q}. 
\end{eqnarray}
\end{cor}
\begin{proof}
 It follows directly from parts (ii) and (iv) of Lemma~\ref{lem:basic} that 
 $[G, H^\varphi]$ is central in $\eta^q(G, H).$ Since under our assumptions 
 $\cal{K} = G \cap H$ is abelian, by relations \eqref{reta1}, \eqref{reta2}, \eqref{reta5} and \eqref{reta6}, we have 
$[\widehat{k}, \widehat{k_1}] = [k^q, (k^{q}_1)^{\varphi}] =  [k^q, (k_1)^{\varphi}]^q = (\widehat{[k, k_1]})^q = 1$ and 
$(\widehat{k})^{g} = \widehat{k} = (\widehat{k})^{h^{\varphi}}$, for all 
 $g \in G, h \in H, k, k_1 \in \cal{K}.$ Hence, the subgroup $K$ is central in 
 $\eta^{q}(G, H)$, too. Consequently, $\Upsilon^q(G, H)$ is central in $\eta^q(G,H).$ The isomorphism \eqref{eq:abelian} is proved in \cite[Theorem 1.24]{Cond}.  \end{proof}

We recall the epimorphism $\rho: \eta^q(G,H) \to GH$, where $g \mapsto g$, $h^\varphi \mapsto h$ and $\widehat k \mapsto k^q$. We write $\theta^q(G, H)$ to denote the kernel, $\ker \rho$. The following result is essentially an adaptation of \cite[Proposition 2.3]{BL} to our context. 

\begin{prop}  \label{prop:epicentral}  
Let $G$ and $H$ be normal subgroups of $L$ and $q$ a non-negative integer. Then,
\begin{itemize}
\item[(a)] The epimorphism $\rho$ induces a homomorphism $\rho' : \Upsilon^q(G, H) \to  G \cap H$ such that $([g, h^{\varphi}]) \rho'  =  [g, h]$ and $(\widehat{k}) \rho'  = k^q$, for all $g\in G$, $h\in H$ and $k \in G \cap H$;
\item[(b)] $[t, h] = [(t)\rho', h^\varphi], \; [g, t] = [g, ((t)\rho')^\varphi]$ and $[t, \widehat{k}] = [((t)\rho')^\varphi, k^q]$, for all $t \in [G, H^\varphi], g \in $G$, h \in H$ and $\widehat{k} \in \cK$; 
\item[(c)] $\mu^q (G,H) := \ker \rho'$ is a central subgroup of $\eta^q(G,H).$ 
\end{itemize}
\end{prop}
\begin{proof}
Item (a) is an immediate consequence of the definitions of $\rho$ and $\Upsilon^q(G, H)$. Item (b) follows by using an induction argument based on Lemma~\ref{lem:basic} (ii), (iv) and commutator calculus. To prove Item (c), we first notice that every element $w \in \Upsilon^q(G,H) = TK$ can be written as a product $w = t \widehat{k}$ where $t \in T = [G, H^\varphi]$ and $k \in \cK$; this follows by an induction argument using defining relations \eqref{reta3} and \eqref{reta4}. Now we have
\begin{equation} \label{eq:central1}
\begin{split}
 [g, w] & = [g, t\wk] \\
        & = [g, \wk] [g, t]^{\wk} \\
	    & = [g, (k^q)^\varphi] [g, ((t)\rho')^\varphi]^{\wk} \quad \text{by Lemma~\ref{lem:basic} (iv), and Item (b)}  \\
        & = [g, (k^q)^\varphi] [g, ((t)\rho')^\varphi]^{k^q} \quad \text{by Relations \eqref{reta3}} \\
        & = [g, (k^q)^\varphi] [g, ((t)\rho')^\varphi]^{(k^q)^\varphi} \quad \text{by Relations \eqref{rel:s2}} \\
	    & = [g, ((t \widehat{k}) \rho')^\varphi]. \\
\end{split}
\end{equation}
With similar arguments we get that $[w, h^\varphi] = [(w)\rho', h^\varphi]$ and $[w, \wk] = [(w)\rho', (k^q)^\varphi]$. Consequently, if $w \in \mu^q(G,H) = \ker \rho'$, then $[g, w] = [w, h^\varphi] = [w, \wk] = 1$, for all $g \in G, \, h^\varphi \in H^\varphi$, and $k \in \cK$. Therefore, $\ker \rho'$ is central in $\eta^q(G,H).$ 
\end{proof} 

\begin{rem} \label{rem:1}
Notice that $\eta^q(G, H) / \Upsilon^q(G,H) \cong G \times H$, while $\eta^q(G, H) / \theta^q (G,H) \cong G H.$ This implies that there is an isomorphism from $\eta^q(G, H) / \mu^q(G,H)$ to a subgroup of $G \times H \times GH$. In addition, 
 $\eta^q(G, H) / [G, H^\varphi] \cong (K/[G, H]) \times G \times H,$ since $[G, H^\varphi] K / [G, H^\varphi]$ is generated by the elements $\wk,$ for $k \in \cK,$ with the relations $\widehat{k k_1} = \wk \widehat{k_1}$ and $[\wk, \widehat{k_1}] = 1,$ according to defining relations \eqref{reta4} and \eqref{reta5}. Set 
$\mu_0^q(G,H) := [G, H^\varphi] \cap \mu^q(G, H)$. We thus obtain the following exact sequence 
\begin{equation} \label{eq:central2}
1 \to \mu_0^q(G,H) \to \eta^q(G, H) \to (K/[G, H]) \times G \times H \times GH. 
\end{equation}   
\end{rem}

Now, let $p$ and $q$ be non-negative integers with $p \geq 1$. Let $\delta: \eta^{pq}(G,H) \to \eta^p(G,H)$ be defined on the generators of $\eta^{pq}(G,H)$ by $(g)\delta:=g$, $(h^\varphi)\delta:=h^\varphi$ and $(\widehat{k})\delta: =\widehat{k^q}$, for all $g \in G$, $h^\varphi \in H^\varphi$ and $k \in {\cal K}$. It is a routine to check that in this way $\delta$ preserves the defining relations of $\eta^p(G,H)$; relations \eqref{reta1}, \eqref{reta2}, \eqref{reta3}, \eqref{reta5} and \eqref{reta6} are easily carried out. However, relation \eqref{reta4} demand tedious calculations and the impatient reader can consult \cite[Theorem 1.22]{Cond}. Thus we obtain a homomorphism from $\eta^{pq}(G,H)$ to $\eta^p(G,H)$. Set $\delta'=\delta\vert_{\Upsilon^{pq}(G,H)}: \Upsilon^{pq}(G,H) \to \Upsilon^p(G,H)$. The next Proposition generalises \cite[Proposition 2.6]{BR} and \cite[Theorem 1.22]{Cond}.

\begin{prop} \label{prop:seq} 
Let $p \geq 1$. There are exact sequences 
	\begin{eqnarray} \label{021}
	\eta^{pq}(G,H) \stackrel{\delta}{\rightarrow} \eta^p(G,H) \to \frac{\cal{K}}{[G, H] \cal{K}^q} \to 1;
	\end{eqnarray}
	\begin{eqnarray} \label{022}
	\Upsilon^{pq}(G,H) \stackrel{\delta'}{\rightarrow} \Upsilon^p(G,H) \to \frac{\cal{K}}{[G, H] \cal{K}^q} \to 1;
	\end{eqnarray}
	In particular, if $q =0$ then we have: 
	\begin{eqnarray} \label{seq:023}
	\eta(G,H) \stackrel{\delta}{\rightarrow} \eta^p(G,H) \to \frac{\cal{K}}{[G, H]} \to 1;
	\end{eqnarray}
	\begin{eqnarray} \label{seq:024}
	\Upsilon(G, H) \stackrel{\delta'}{\rightarrow} \Upsilon^p(G,H) \to \frac{\cal{K}}{[G, H]} \to 1.
	\end{eqnarray}
\end{prop}
	\begin{proof} According to the definition of $\delta$  we have 
	$$\Imm(\delta)=\left\langle g, h^\varphi, \widehat{k^q} \, \vert \, g \in G, \text{ } h \in H \text{ and } k \in {\cal K} \right\rangle.$$ 
	Thus, it follows from Lemma \ref{lem:basic} that $\Imm(\delta)$ is a normal subgroup of $\eta^p(G,H)$. Now, as already observed in Remark~\ref{rem:1}, $\eta^p(G,H)/\Imm(\delta)$ is generated by the cosets of the elements $\widehat k \in \widehat {\cal K}$ with the relations 
	$\widehat{k k_1}=\widehat k \widehat{k_1}$. Furthermore, as $\widehat{k^q} \in \Imm(\delta)$ for all $k \in {\cal K}$, it follows that $(\widehat k)^q \equiv 1 \pmod { \Imm(\delta)}$. This proves \eqref{021}. The sequence \eqref{022} is essentially \cite[Theorem 6, (ii)]{Ellis} and follows by a similar argument as above, since $\Imm(\delta')$ is also normal in $\Upsilon^q(G, H)$. The sequences \eqref{seq:023} and \eqref{seq:024} follow at once, respectively from \eqref{021} and \eqref{022} with $q = 0$. This completes the proof. \end{proof}
	
The next result shows that the derived group $\eta^q(G, H)'$ has the same formal structure for all $q \geq 0.$ In order to avoid any confusion, we write $\cT(G,H)$ for the subgroup $[G, H^\varphi] \leqslant \eta(G,H)$ (case $q=0$), which is isomorphic with the non-abelian tensor product $G \otimes H$ for all compatible actions of one group upon another. In many places in this paper we write $T$ for the subgroup $[G, H^\varphi] \leqslant \eta^q(G,H),$ so that the q-tensor product $G \otimes^q H \cong T K \leqslant \eta^q(G,H).$ If $q=0$ then $\cT(G,H) = T.$  

\begin{prop} \label{prop:structure2}
Let $G$ and $H$ be normal subgroups of a group $L.$ Then, for all $q \geq 0,$
$\eta^q(G, H)' = [G, H^\varphi] \cdot G' \cdot (H')^\varphi.$ In particular, if $q=0$ then the non-abelian tensor product 
$G \otimes H$ is embedded in $\eta(G, H)'.$
\end{prop}
\begin{proof}
With the above discussion, we can write $\eta^q(G,H) = TK G H^\varphi$, according to 
\eqref{eq:structure}. Now, $T$ and $TK$ are normal subgroups of $\eta^q(G,H),$ while 
$[G, H^\varphi] = T.$ By defining relations \eqref{reta5} we find that 
$K' \leqslant T.$ Additionally, from Lemma~\ref{lem:basic} (v) we obtain that $[K, G]$ and $[K, H^\varphi]$   
are both contained in $T.$ Therefore, $\eta^q(G, H)' = [T K G H^\varphi, T K G H^\varphi] = T \, G'\, (H')^\varphi = 
 [G, H^\varphi] G' (H')^\varphi.$ 
\end{proof}

\section{Some Closure Properties for  $\eta^q(G, H)$}

A number of authors have studied some closure properties such as finiteness, solubility, polycyclicity and nilpotency, among others, of the non-abelian tensor product of groups and of related constructions (cf. \cite{DLT,LT,LO,Moravec1,Nak,V}). In the context of $\nu(G), \; \eta(G,H)$ and $\nu^q(G)$, such closure properties were studied for instance in 
\cite{BFM, BM, Rocco, Rocco1, BR, DR, BNRa}. 
In this section we extend these considerations to the scope of $\eta^q(G, H), \; q \geq 0$. We will consider the following question: Let $\mathfrak{X}$ be a class of groups. If $G, H$ are normal subgroups of a certain group $L$ such that $G$ and $H$ belong to $\mathfrak{X}$, then does $\eta^q(G, H)$ belong to $\mathfrak{X}$?

Recall that a class $\mathfrak{X}$ of groups is called a \textit{Schur class} if for any group $G$ such that the factor group $G/Z(G)$ belongs to $\mathfrak{X}$, also the derived subgroup $G'$ is a $\mathfrak{X}$-group. Thus, the famous Schur's theorem just states that finite groups form a Schur class. Other interesting classes of groups (e.g., finite $\pi$-groups, locally (finite $\pi$-groups), polycyclic groups, polycyclic-by-finite groups) are Schur classes. The classical reference to this matter is \cite{Rob.finiteness}.  

In \cite{BL} Brown and Loday  proved that if $G$ is a finite $\pi$-group, then the non-abelian tensor square $G \otimes G$ is a finite $\pi$-group; in particular, $\nu(G)$ is a finite $\pi$-group. Ellis \cite{Ellis} proved the finiteness of $G \otimes H$ when $G$ and $H$ are finite groups. Moravec \cite{Moravec} showed that if $G$ is a locally (finite $\pi$-groups), then the so is $\nu(G)$. In \cite{LO}, Lima and Oliveira proved that if $G$ is polycyclic-by-finite, then so is $\nu(G)$. Here we extend these results to the scope of  $\eta^q(G,H), \; q \geq 0,$ and give elementary proofs of them by using only the structural properties discussed in Section 2 and the definition of a Schur class, based on Proposition~\ref{prop:epicentral} and Remark~\ref{rem:1}. 

We write $\mathfrak{F}_{\pi}$ to indicate the class of finite $\pi$-groups and $L\mathfrak{F}_{\pi}$ for the class of locally (finite $\pi$-groups), where $\pi$ is a set of primes. To ease reference we state Lemma~\ref{lem:Schur2}, which extends Schur's theorem to the class $L\mathfrak{F}_{\pi}$.

\begin{lem} \label{lem:Schur2} 
Let $G$ be any group. If $G/Z(G) \in L\mathfrak{F}_{\pi}$ , then $G' \in L\mathfrak{F}_{\pi}.$ 
\end{lem}

In the next theorem we establish some closure properties on $\eta^q(G,H).$  

\begin{thm} \label{thm:finiteness} 
 Let $G$ and $H$ be normal subgroups of a group $L$ and let $q$ be a non-negative integer. Then, 
	\begin{itemize}
	\item[(i)] If $G, H \in \mathfrak{F}_{\pi}$, then $\eta^q(G, H) \in \mathfrak{F}_{\pi};$ 
	\item[(ii)] If $G, H \in L\mathfrak{F}_{\pi}$, then $\eta^q(G, H) \in L\mathfrak{F}_{\pi};$ if furthermore $G$ and  $H$ have finite exponents, then the exponent of $\eta^q(G,H)$ can be bound in terms of $q$, $\exp G$ and $\exp H;$
	\item[(iii)] If $G$ and $H$ are soluble groups of derived lengths $l_1$ and $l_2$, respectively, then $\eta^q(G,H)$ is also soluble, of at most derived length $l_1 + l_2 + 1$;
	\item[(iv)] If $G$ and $H$ are nilpotent groups of nilpotency classes $c_1$ and $c_2$, respectively, then $\eta^q(G,H)$ is nilpotent of at most class $c_1 + c_2 + 1$; 
	\item[(v)] If $G$ and $H$ is polycyclic-by-finite, then so is $\eta^q(G,H)$. 
	\end{itemize}
\end{thm}	
	\begin{proof} 
 (i). If $G, H \in \mathfrak{F}_{\pi}$ then clearly $(\cK/[G,H]) \times G \times H \times GH \in \mathfrak{F}_{\pi}.$ Thus,  
$\eta^q(G, H)/ \mu_0^q(G, H) \in \mathfrak{F}_{\pi}$, by Remark~\ref{rem:1}, sequence \eqref{eq:central2}. By Schur's theorem (\cite[10.1.4]{Rob}), $\eta^q(G,H)'$ is finite and $\exp(\eta^q(G,H)')$ divides $|G|^2 |H|^2$, that is, 
$\eta^q(G,H)' \in {\fF}_{\pi}.$ Since, by Proposition~\ref{prop:structure2}, $\mu_0^q(G,H) \leqslant [G, H^\varphi]  \leqslant \eta^q(G,H)'$, we get that $\mu_0^q(G,H) \in {\fF}_{\pi}$ and so, $\eta^q(G, H) \in \mathfrak{F}_{\pi}$.

\noindent
(ii). Using a similar argument as in Part (i) and Lemma~\ref{lem:Schur2} we find that both $\eta^q(G,H) / \mu_0^q(G,H)$ and $\mu_0^q(G,H)$ are locally (finite $\pi$-groups). Therefore, $\eta^q (G, H) \in L\mathfrak{F}_{\pi},$ by Reidemeister-Schreier's theorem (\cite[6.1.8]{Rob}). If in addition $G$ and $H$ have finite exponents, then by \cite[Corollary 5]{Moravec} $\exp \eta(G,H)$ can be bound in terms of $\exp G$ and 
$\exp H.$ Thus, by sequence \eqref{seq:023} in Proposition~\ref{prop:seq} we see that the same is true for $\eta^q(G,H),$ with the additional restriction that, due to Lemma~\ref{lem:basic} (vi), such upper bound may also involve $q$ in $\Imm \delta.$

\noindent
Part (iii) follows directly from Proposition~\ref{prop:epicentral}, by the fact that $G \times H \times G H$ is soluble of at most derived length $l_1 + l_2$, while $\mu^q(G,H)$ is abelian. 

\noindent
(iv). Analogously, by Fitting's theorem \cite[5.2.8]{Rob} $G \times H \times GH$ is nilpotent of class at most $c_1 + c_2$, while $\mu^q(g,H)$ is central in $\eta^q(G,H)$.   

\noindent (v). Again, we have that $\eta^q(G, H)/\mu_0^q (G,H)$ is polycyclic-by-finite and $\mu_0^q(G,H) \leqslant Z(\eta^q(G,H))$. Therefore,  $\eta^q(G, H)'$ is polyclicic-by-finite. It suffices to show that $\mu_0^q(G, H)$ is polycyclic. Let 
$M \trianglelefteq \eta^q(G, H)'$ be a polycyclic normal subgroup of finite index. Then $M \cap \mu_0^q(G,H)$ is a polycyclic normal subgroup of finite index of the abelian group $\mu_0^q(G,H)$. Therefore $\mu_0^q(G,H)$ is polycyclic. 
The proof is complete.  
\end{proof}
\begin{rem} \label{rem:nilclass}
Notice that if $G = H = L$ is a soluble group of derived length $l$ (respectively, nilpotent of class $c$), then the bounds in parts (iii) and (iv) of Theorem~\ref{thm:finiteness} become $l + 1$ (respectively, $c+1$), according to \cite[Theorem 2.8]{BR}.
\end{rem}
\section{A polycyclic presentation for the $q$-tensor square of the dihedral group $D_n$, $q$ odd} 

In this Section we restrict our attention to the group $\nu^q(G)$ (that is, the group $\eta^q(G,G)$ when $G = H = L$), particularly on the computation of the $q$-tensor square of the dihedral groups $D_n$ when $q$ is odd. We begin with a brief description of an algorithm derived in \cite{DR} for  computing polycyclic presentations for $\nu^q(G)$, when $G$ is a polycyclic group given by a consistent polycyclic presentation. 

Let $G$ be a polycyclic group defined by a consistent polycyclic presentation $F/R$, where  $F$ is a free group generated by $g_1,...,g_n$.  The relations of a consistent polycyclic presentation $F/R$ are 
(cf. \cite [Section 9.4]{Sims}): 

\begin{itemize}
	\item[]\hspace{2cm}$ g_i^{e_i}= g_{i+1}^{\alpha_{i,i+1}}... g_n^{\alpha_{i,n}}$ for $i \in I$,
	\item[]\hspace{2cm}$g_j^{-1}g_ig_j=g_{j+1}^{\beta_{i,j,j+1}}...g_n^{\beta_{i,j,n}}$ for $j<i$,
	\item[]\hspace{2cm}$g_jg_ig_j^{-1}=g_{j+1}^{\gamma_{i,j,j+1}}...g_n^{\gamma_{i,j,n}}$ for $j<i$ e
	$j \notin I$,
\end{itemize}
for some  $I\subseteq \{1,...,n\}$, some  exponents  $e_i \in \N$ with $i\in I$, $\alpha_{i,j}$, $\beta_{i,j,k}$, $\gamma_{i,j,k} \in \Z$ and for all $i$, $j$ and  $k \in\{1,...,n\}$. Recall that this presentation is {\it refined} if all $i\in I$ are prime numbers.

Following  \cite{DR}, we write the relations of $G$ as relators of the form $r_1,...,r_l$, where  every relator  $r_j$ is a word in the  generators $g_1,...,g_n$,  $r_j=r_j(g_1,...,g_n)$. Let
\[
E_q(G):= \dfrac{F}{R^q[F,R]}, 
\]
which is a $q$-central extension  of  $G$.  

A presentation for the group $E_q(G)$ can be obtained according to the following construction (see \cite{DR}): for each relator $r_i$ we introduce a generator $t_i$; $E_q(G)$ is then the group generated by $g_1,...,g_n,t_1,...,t_l$ subject to the relators:
\begin{itemize}
	\item[(1)] $r_i(g_1,...,g_n)t_i^{-1}$ for $1\leq i\leq l$,
	\item[(2)] $[t_i,g_j]$ for $1\leq j \leq n$, $1\leq i \leq l$,
	\item[(3)] $[t_i,t_j]$ for $1\leq j < i \leq l$,
	\item[(4)] $t_i^q$ for $1\leq i\leq l$.
\end{itemize}

This is a polycyclic presentation of $E_q(G)$, possibly inconsistent (see \cite{DR}). The consistency relations can be evaluated in the consistent polycyclic presentation of $E_q(G)$ using a collection system from  left  to right, given by
\begin{itemize}
	\item[(1)] $r_i(g_1,...,g_n)t_1^{q_{i1}}...t_l^{q_{il}}$ for $1\leq i\leq l$,
	\item[(2)] $[t_i,g_j]$ for $1\leq j \leq n$, $1\leq i \leq l$,
	\item[(3)] $[t_i,t_j]$ for $1\leq j < i \leq l$,
	\item[(4)] $t_i^{d_i}$ for $1\leq i\leq l$ with $d_i\mid q$.
\end{itemize}

\begin{prop} \cite[Proposition 4.2]{DR} \label{prop} 
$G \wedge^q G=(E_q(G))'(E_q(G))^q$. Furthermore,  $(E_q(G))'(E_q(G))^q$ is the subgroup \[\left\langle [g_i, g_j]^{\epsilon}, g_k^q \; \vert \; 1\leq i < j \leq n, 1 \leq k \leq n \right\rangle, \] 	
where $\epsilon = 1$ if $G$ is finite and $\epsilon = \pm 1$ otherwise.
\end{prop}

Hence, to obtain a presentation  for the  $q$-exterior square of a polycyclic group $G$ defined by a consistent polycyclic presentation,  we apply the standard methods to determine presentations for subgroups of polycyclic groups (see \cite{EN}).

Now, let us consider the  dihedral group  $G=D_{n}$ given by the consistent polycyclic presentation 
\[
D_{n}=<g_1,g_2 \ | g_1^2=1, g_1^{-1}g_2g_1=g_2^{n-1}, g_2^n=1>.
\] 

We will compute the $q$-exterior square of  $D_{n}$, $q$-odd. To this end we begin with the following proposition. Recall that a group $G$ is called $q$-perfect if $G = G' G^q$. 

\begin{prop}
For $q$ odd,  $D_{n}$ is $q$-perfect.
\end{prop}
\begin{proof}  Let $G := D_{n}$. Then we have \[G' G^q = <g_1^q, g_2^q, [g_1, g_2]>=<g_1,g_2^q,g_1^2>.\] Since  $q$ is odd, $qx+2y=1$ for some $x,y \in \Z$. Therefore, $g_2=(g_2^q)^x (g_2^2)^y \in K'K^q.$
\end{proof}

Notice that if  $G$ is $q$-perfect then by Lemma~\ref{lem:basic}, $\Delta^q(G) = 1$ and hence $G \otimes^q G \cong G \wedge^q G$. In this case $G \otimes^q G$ can  be computed  using  Proposition~\ref{prop}. Now,  
\[
E_q(D_{n})=<g_1,g_2,t_1,t_2,t_3 | g_1^2=t_1,  g_1^{-1}g_2g_1=g_2^{n-1}t_2, g_2^3=t_2, t_i \ \ q\text{-central}, \ i=1,2,3>.
\] 
Testing the consistency of this presentation we obtain the  unique relation:  
\begin{eqnarray} \label{consitencia} 
t_2^n t_3^{n-2}=1.
\end{eqnarray} 
Computing  a consistent polycyclic presentation  for  
$W=E_q(G)'E_q(G)^q \cong D_{n} \wedge^q D_{n}$ we obtain 
$W= \langle [g_1,g_2],g_1^q,g_2^q \rangle.$ Since  $[g_1,g_2]=g_2^{-(n-2)}t_2^{-1}=(g_2^{n-2}t_2)^{-1}=[g_2,g_1]^{-1}$, we find $W=<g_2^{n-2}t_2, g_1^q,g_2^q>.$

Routine computations give us the following result, where we  write $(a,b)$ for $\gcd(a,b)$, $[a,b]$ for $\lcm(a,b)$ and $o(g)$ for the order of the element $g$:  
\begin{prop}
Let $q$ be an odd integer. Then, in $E_q(D_n)$ we have: 
	\[
	o(g_1)=2q, \ o(g_2)=[n,q], \ o(t_1)=q,\ o(t_2)=\dfrac{q}{(n-2,q)} \text{ and } o(t_3)=\dfrac{q}{(n,q)}.
	\]		
\end{prop}

\begin{prop}
	For $q$ odd, $D_{n} \otimes^q D_{n} \cong D_{n}$ and  $H_2(D_{n},\Z_q)=\{1\}.$
\end{prop}
\begin{proof} As above, we have that $D_{n}\otimes^q D_{n} \cong W =\langle g_2^{n-2}t_2, g_1^q,g_2^q \rangle.$ We need to show that $W \cong D_{n}$. Since $q$ is odd,  $qx+2y=1$ for some $x,y \in \Z$. Set $h:=(g_2^q)^x(g_2^{n-2}t_2)^{-y}=g_2(t_2t_3)^{-y} \in  W$. Thus, $g_1^{-q}hg_1^q=h^{-1}$ and $h^n=1$. Moreover, since $(q,y)=1$ and $o(g_2)\geq n$, it follows that $g_2^l\neq 1$ and $g_2^l\neq t_3$ for all $l<n$. Therefore, $o(h) \leq n$ and thus, $o(h)=n$. On the other hand, setting $H:=\langle h \rangle$ we find that  $g_2^q=h^q \in H$, $(g_2^{n-2}t_2)^{-y}=(g_2^q)^{-x}h\in H$ and 
$g_2^{n-2}t_2=\{(g_2^{n-2}t_2)^{-y}\}^{-b}(g_2^q)^{-2a}\in H$, where $a,b$ are integers such that $qa+yb=1$. Therefore,
\[
W = \langle g_1^q, h \; \mid \; (g_1^q)^2=1,  \ g_1^{-q}hg_1^q=h^{-1}, \ h^n=1 \rangle \cong D_{n}. 
\]
\end{proof}

\section{Exponents of the $q$-tensor square of nilpotent groups of class $\leq 3$} 

Moravec \cite{Moravec1} gives an estimate  for $\exp(G \otimes G)$  in terms of  $\exp(G)$  and $\exp(G \wedge G)$ and he observed that  for finite metabelian groups,  $\exp (G \wedge G)$ divides $(\exp G)^2 \exp G'$; consequently, 
$\exp(G \otimes G)$ divides $(\exp G)^3 \exp G'$. For finite nilpotent groups of class $\leq 3$ he proved that 
$\exp(G \otimes G)$ divides $\exp G$ (cf. \cite[Theorem 2]{Moravec1}). In this section we show that this upper bound can be extended to the $q$-tensor square of finite nilpotent groups of class $\leq 3,$ $q \geq 0.$ 
\begin{lem} \label{lem:powers}
Let $G$ be a nilpotent group of class $\leq 3$ and let $q \geq 1.$ Then
\begin{itemize}
    \item[(i)] $[K, \gamma_3(\nu^q(G))] = 1;$
    \item[(ii)] if $t = \prod_{i =1}^r [x_i, y_i^\varphi]^{\epsilon_i}$ is an arbitrary element in $T = [G, G^\varphi],$ where $r \geq 1$ and $\epsilon_i = \pm 1$ for $i = 1, \ldots, r,$ then, for all $\widehat{k} \in \widehat{\cK}$ we have
    \[ 
    [t, \wk] = \prod_{i=1}^r [x_i, y_i^\varphi, \wk]^{\epsilon_i} = \prod_{i=1}^r [x_i, y_i^\varphi, k^q]^{\epsilon_i} = [t, k]^q;
    \]
       \item[(iii)] for all positive integers $n$ and $t, \wk$ as in part (ii), we have,
       \[
       (t \wk)^n = t^n [t,k^{q}]^{-  \binom{n}{2}} (\wk)^n = t^n [t,k^{- q \binom{n}{2}}] (\wk)^n
       \]
\end{itemize}
\end{lem}
\begin{proof}
  (i). By \cite[Proposition 2.7]{BR} we know that $\gamma_j(\nu^q(G)) = [\gamma_{j-1}(G), G^\varphi] \gamma_j(G) \gamma_j(G^\varphi)$, for all $j\geq 2$. This implies that $\nu^q(G)$ has nilpotency class at most 4. Now let $\wk \in \widehat{\cK}$ be any generator of $K$. From defining relations \eqref{reta1}, \eqref{reta2} and  Lemma~\ref{lem:basic} (iv), we see that conjugation of $\wk$ by any  commutator $[x^{\alpha}, y^{\beta}, z^{\gamma}] \in \gamma_3(\nu^q(G))$, where $\alpha, \beta, \gamma \in \{1, \varphi \},$ is the same as conjugating $\wk$ by the commutator  $[x, y, z] \in \gamma_3(G).$ This shows that $\gamma_3(\nu^q(G))$ centralizes $K$ if $G$ has nilpotency class $\leq 3.$ 
  
  \noindent
  (ii). The first equality follows from  commutator calculus and induction on $r$, since $[\gamma_3(\nu^q(G)), \gamma_2(\nu^q(G))] \leq  \gamma_5(\nu^q(G)) = 1,$ as $\nu^q(G)$ has class $\leq 4.$ The second equality follows from the identity $[x, y^\varphi, \wk]= [x, y^\varphi, k^q],$ according to Lemma~\ref{lem:basic} (iv), while the last one is obtained by the way back, making use of part (i) to write $[x_i, y_i^{\varphi}, k^q] = [x_i, y_i^{\varphi}, k]^q$ for $i = 1, \ldots, r.$
  
  \noindent
  (iii). We expand $(t \wk)^n$ by induction on $n$, collecting commutators in the middle, to get 
  \begin{equation} \label{eq:powers}
  (t \wh)^n = t^n \prod_{i=1}^{n-1}([t, (\wk)^{ -i}]^{t^{n-1-i}}) (\wk)^n.
  \end{equation}
  Since $[t, (\wk)^{-i}] \in \gamma_3(\nu^q(G))$, we see that $[t, (\wk)^{ -i}]^{t^{n-1-i}} = [t, (\wk)^{ -i}]$, $i = 1, \ldots,{n-1}.$ In addition, by parts (i) and (ii), $[t, (\wk)^{ -i}] = [t, (\wk)]^{ -i} = [t, k^q]^{ -i} = ([t, k]^{-q})^i.$ Consequently, $\prod_{i=1}^{n-1}([t, (\wk)^{ -i}]^{t^{n-1-i}}) = [t, k]^{-q \binom{n}{2}}$ which, by part (i) and induction, is also equal to $[t, k^{-q \binom{n}{2}}]$ since $[t, k] \in \gamma_3(\nu^q(G)).$ This completes the proof. 
\end{proof}

\begin{thm} \label{thm:expclass3}  Let $G$ be a finite nilpotent group of class $\leq 3$ with $\exp G = n$ and let $q \geq 0.$  Then, 
	\begin{itemize}
	\item[(i)] $\exp (G \otimes^{q} G)$ divides  $\exp G$ if either $n$ is odd or $4 \mid q;$ 
	\item[(ii)] $\exp (G \otimes^{q} G)$ divides $2 \exp G$, otherwise.
	\end{itemize}
\end{thm}  
\begin{proof}
We use the isomorphism $G \otimes^q G \cong \Upsilon^q(G) = T K$, where $T = [G, G^\varphi]$, to work inside $\nu^q(G).$ 
As already observed in the proof of Proposition~\ref{prop:epicentral}, by defining relations \eqref{reta3} and \eqref{reta4} we see that $\widehat{\cK}$ provides a right transversal for $T$ in $T K$, that is, 
$T K = \{ t \wk \; \mid \; t \in [G, G^\varphi], k \in \cK \}.$ So, all we need is to control the orders of an arbitrary element $t \wk \in \Upsilon^q(G).$ 

\noindent
(i). Firstly we consider the case $q = 0.$ Here we have $\nu^0(G) = \nu(G)$ and $G \otimes^0 G = G \otimes G,$ the non-abelian tensor square of $G$. So, this case is dead by \cite[Theorem 2]{Moravec}: $\exp (G \otimes G)$ divides $\exp G.$ 
Thus, suppose $q \geq 1$ and let $n : = \exp G.$ 
By part (iii) of Lemma~\ref{lem:powers} we have 
\begin{equation} \label{eq:final}
(t \wk)^n = t^n [t, k^{- q \binom{n}{2}}] (\wk)^n = (\wk)^n,
\end{equation}
if $n$ is odd or if $q$ is even. Notice that $t \in [G, G^\varphi]$ and so, $t^n = 1$ for all $n$, by the case $q=0$ and \eqref{seq:023}). Now, an induction on $n$ using relation \eqref{reta4} gives 
\[
1 = \widehat{(k^n)} = (\wk)^n [k, k^\varphi]^{- \binom{q}{2} \binom{n}{2}},
\]
which implies that $(\wk)^n = [k, k^\varphi]^{\binom{q}{2} \binom{n}{2}}.$ Thus, if $n$ is odd or if $4 \mid q,$ then $(\wk)^n = 1$ and, consequently, $(t \wk)^n = 1.$

\noindent
(ii). In the case $n$ even and $4 \nmid q$ then certainly we get, from \eqref{eq:final} and Lemma~\ref{lem:powers} (i), 
\[
(t \wk)^{2n} = ([t, k^{- q \binom{n}{2}}] (\wk)^n)^2 =  [t, k^{- q \binom{n}{2}}]^2 ((\wk)^n)^2 = [t, k^{-q n(n-1)}] [k, k^\varphi]^{\binom{q}{2} n (n-1)} = 1.
\]
This completes the proof. 
\end{proof}

\begin{ex} The third of the following simple examples borrowed from \cite[Theorem 3.1 and Table 1]{BR} shows that the bound in part (ii) of Theorem~\ref{thm:expclass3} can be attained in the simplest situation, of a cyclic group.   
\begin{enumerate}
\item $D_4 \otimes^4 D_4 \cong C_2^5 \times C_4$; 
\item $Q_8 \otimes^4 Q_8 \cong C_2^4 \times C_4^2$; 
\item $C_n \otimes^q C_n \cong C_{2n} \times C_s$ if  $q, n  \equiv 2 \pmod 4$ and $\gcd(q,n) = 2s$.
\end{enumerate}
\end{ex}


\renewcommand{\refname}{REFERENCES}

\end{document}